\documentclass[12pt, a4paper]{amsart}

\usepackage{amsfonts,amsmath,amssymb, amscd,fullpage}
\usepackage{stmaryrd}
\usepackage[all]{xy}

\newtheorem{theorem}{Theorem}[section]

\newtheorem{proposition}[theorem]{Proposition}

\newtheorem{question}[theorem]{Question}

\theoremstyle{definition}

\theoremstyle{remark}
\newtheorem{remark}[theorem]{Remark}

\newcommand\pf{\begin{proof}}
\newcommand\epf{\end{proof}}

\newcommand\Cc{\mathcal{C}}
\newcommand\Dd{\mathcal{D}}
\newcommand\Mm{\mathcal M}

\newcommand\ext{\mathrm{Ext}}

\newcommand\cd{\mathrm{cd}}
\newcommand\pd{\mathrm{pd}}

\DeclareMathOperator{\gldim}{gldim}

\numberwithin{equation}{section}

\title{Monoidal invariance of the cohomological dimension of  Hopf algebras: the finite case}

\author{Julien Bichon}
\address{ Universit\'e Clermont Auvergne, CNRS, LMBP, F-63000 CLERMONT-FERRAND, FRANCE}
\email{julien.bichon@uca.fr}


\begin{document}

\begin{abstract}
	A consequence of the recent work of Ren and Zhu on Gorenstein projective dimensions of modules over Hopf algebras is that if $A$ and $B$ are Hopf algebras with bijective antipodes having equivalent linear tensor categories of comodules and both having finite global dimensions, then their global dimensions coincide.
	In this note we provide a direct proof of this result, without using  Gorenstein projective dimensions, and we notice that the assumption on the bijectivity of the antipodes can be removed.  
\end{abstract}

\maketitle


\section{introduction}

The following question, which enters into the general classical problem of  determining which properties (ring-theoretical, homological...) of an algebra are preserved by ``deformation", was originally asked in \cite{bi16}:

 \begin{question}\label{ques0}
	If $A$ and $B$ are Hopf algebras having equivalent linear tensor categories of comodules, do we have $\gldim(A)= \gldim(B)$? 
\end{question}

After a number of positive partial results \cite{wyz, bi22}, the recent works of R. Zhu \cite{zhu} and of W. Ren and R. Zhu \cite{renzhu} provide definitive answers. First, Zhu \cite{zhu} has shown that the answer to Question \ref{ques0} is no in general, by exhibiting Hopf algebras $A, B$ with $\Mm^A \simeq^\otimes\Mm^B$ (this is the symbol we use to indicate that $A$ and $B$ have equivalent linear tensor categories of comodules) and  $\gldim(A)=1$ and $\gldim(B)=\infty$.  The next question then is if there are still counterexamples if the global dimensions both are assumed to be finite, and the paper \cite{renzhu} by Ren and Zhu, which deals with Gorenstein projective dimensions of modules over Hopf algebras, has the following positive consequence: 

	\begin{theorem}\label{thm:cd0}
		Let $A$, $B$ be Hopf algebras with bijective antipodes such that  $\Mm^A \simeq^\otimes\Mm^B$. If $\gldim(A)$ and $\gldim(B)$ are finite, then $\gldim(A)=\gldim(B)$.
	\end{theorem}

Indeed \cite[Theorem 3.2]{renzhu} proves that  if $A$, $B$ have bijective antipodes and $\Mm^A \simeq^\otimes\Mm^B$, then the Gorenstein global dimensions of $A$ and $B$ coincide, and hence, since when the global dimension of an algebra is finite, it coincides with the Gorenstein global dimension \cite{holm},  Theorem \ref{thm:cd0} follows.

The aim of the present note is to provide a proof of Theorem \ref{thm:cd0} that does not use or mention Gorenstein projective dimensions. We also notice that the assumption of bijectivity of the antipode can be removed.

Througout the paper we work over a fixed base field $k$,  and we assume that the reader has familiarity with the theory of Hopf algebras, of which we use some standard notation.

\section{The key result}

We begin by a general result on the preservation of projective dimensions by functors under certain assumptions. As in \cite{chewen},  if $\Cc$ and $\Dd$ are categories, and $F : \Cc \to \Dd$ and $G : \Dd \to \Cc$ are functors, the notation $F : \Cc \rightleftarrows \Dd : G$ means  that $G$ is right adjoint to $F$.  Also, as in \cite{chewen}, if $\mathcal S$ is a class of  objects of  an abelian category $\Cc$, the notation  ${\rm add}(\mathcal S$) stands for the full subcategory of $\Cc$ whose objects are the finite directs sums of directs summands of objects of $\mathcal{S}$ (hence if $\mathcal{S}$ is stable under direct sums, then ${\rm add}(\mathcal S)$ simply consists of directs summands of objects of $\mathcal{S}$). The notation ${\rm Proj}(\Cc)$ stands for the class of projective objects of $\Cc$.

The following result can be obtained as  a straightforward consequence of \cite[Proposition 3.7]{chewen} (if we make the additional assumption that the functor $F$ below is faithful), but we give a proof that does not use Gorenstein projective dimension.

\begin{proposition}\label{prop:pd}
	Let $\Cc, \Dd$ be abelian categories having enough projective objects, and let
	$F : \Cc \rightleftarrows \Dd : G$ be pair of adjoint functors. Assume that $F$ and $G$ are exact  and that ${\rm add}(G({\rm Proj}(\Dd)))= {\rm Proj}(\Cc)$. Then, for any object $X$ in $\Cc$ such that $\pd_\Cc(X)$ is finite, we have $\pd_\Cc(X)= \pd_\Dd(F(X))$.
\end{proposition}

\begin{proof}
Since $G$ is exact and $F$ is left adjoint to $G$, a classical argument shows that $F$ preserves projective objects, and since $F$ is exact, it easily follows that for any object $X$ of $\Cc$, we have $\pd_\Dd(F(X))\leq  \pd_\Cc(X)$. Notice also that by  \cite[Corollary 1]{adarie}, we have for any object $Q$ in $\Dd$
\[\ext^*_\Dd(F(X), Q)\simeq \ext^*_\Cc(X,G(Q))\]
which also proves that $\pd_\Dd(F(X))\leq  \pd_\Cc(X)$.
Assume now that $\pd_\Cc(X)=m$ is finite. A classical argument using the long $\ext^*_\Cc$ exact sequence then shows that there is a projective object $P$ in $\Cc$ such that $\ext^m_\Cc(X, P)\not=\{0\}$. The last assumption ensures that there exist $Q$ a projective object in $\mathcal D$ and an object $P'$ of $\Cc$ such that
$G(Q)\simeq P \oplus P'$
Hence we have
\[\ext^m_\Dd(F(X), Q)\simeq \ext^m_\Cc(X,G(Q))\simeq \ext^m_\Cc(X,P\oplus P')\simeq \ext^m_\Cc(X,P)\oplus \ext^m_\Cc(X, P')\not=\{0\}\]
which shows that $\pd_\Dd(F(X))\geq m = \pd_\Cc(X)$, and finishes the proof.
\end{proof}

\section{Cohomological dimension of Hopf-Galois objects}

We now provide the proof of Theorem \ref{thm:cd0}, following a path that is now standard, see \cite{yu,wyz,bi22}, and consists of comparing the cohomological dimension of a Hopf algebra with the one of its Hopf-Galois objects, and use the fact that monoidal equivalences as in Question \ref{ques0} always arise from bi-Galois objects \cite{sc1}.

Let $A$ be a Hopf algebra and let $R$ be a left (resp. right) $A$-comodule algebra. Recall that the \textsl{canonical functor associated to $R$} is the functor
\begin{alignat*}{3}
	F_R^l : {_A \mathcal M} &\longrightarrow {_R\mathcal M}_R  \quad \quad \text{(resp.} \	F_R^r :  &\mathcal M_ A&\longrightarrow {_R\mathcal M}_R \\
	V &\longmapsto V \odot R \quad & V &\longmapsto R \odot V)
\end{alignat*}
where $V\odot R$ (resp $R\odot V$) is $V\otimes R$ (resp. $R\otimes V$) as a vector space, and has $R$-bimodule structure given by
\begin{align*}x\cdot (v\otimes y)\cdot z &= x_{(-1)}\cdot v \otimes x_{(0)}yz, \quad x,y,z \in R, \ v\in V
\\ 
(\text{resp.} \ x\cdot (y\otimes v)\cdot z &= xyz_{(0)} \otimes v\cdot z_{(1)} \quad x,y,z \in R, \ v\in V)
\end{align*}
The left (resp. right) $A$-comodule algebra $R$ is said to be a left (resp. right) $A$-Galois object if  $R$ is  non-zero and the canonical map
\begin{alignat*}{3}
	{\rm can}_l : R\otimes R & \longrightarrow A \otimes R  \quad \quad\quad \quad  (\text{resp.} \ {\rm can}_r :& R\otimes R & \longrightarrow R \otimes A \\
	x\otimes y &\longmapsto x_{(-1)} \otimes x_{(0)}y,  \quad &x\otimes y &\longmapsto xy_{(0)} \otimes y_{(1)},)
\end{alignat*}
is bijective. See \cite{scsurv} for a general overview of the theory of Hopf-Galois objects. 

For a left (resp. right) $A$-Galois object $R$, 
the map $ \kappa_l : A \to R\otimes R$ defined by $\kappa_l (a) = {\rm can}_l^{-1}(a\otimes 1)$ (resp. $ \kappa_r : A \to R\otimes R$ defined by $\kappa_r (a) = {\rm can}_r^{-1}(1\otimes a)$) defines then an algebra map $ A \to R\otimes R^{\rm op}$ (resp. $ A \to R^{\rm op}\otimes R$), for which we use a Sweedler type notation $\kappa_l(a)=a^{[1]} \otimes a^{[2]}$ (resp. $\kappa_r(a)=a^{\langle 1 \rangle } \otimes a^{\langle 2\rangle }$). 
The above algebra map defines in the usual way the restriction functor
\[\kappa_{l,*} : {_R\mathcal M}_R \to {_A\mathcal M}, \quad (\text{resp.} \ \kappa_{r,*} : {_R\mathcal M}_R \to {\mathcal M}_A ,)\]
that associates to an $R$-bimodule $M$ the left (resp. right) $A$-module $\kappa_{l,*}(M)$ (resp. $\kappa_{r,*}(M)$) having $M$ as underlying vector space and left (resp. right) $A$-action, called the Miyashita-Ulbrich action, defined by $a\cdot x = a^{[1]}.x.a^{[2]}$  (resp.  $x\cdot a = a^{\langle 1\rangle }.x.a^{\langle2\rangle }$).

\begin{proposition}\label{prop:adjcan}
	Let $A$ be a Hopf algebra and let $R$ be a left (resp. right) $A$-Galois object. Then 
	\[F_R^l : {_A\mathcal M}  \leftrightarrows   {_R\mathcal M}_R : \kappa_{l,*} \quad (\text{resp.}  \ F_R^r : {\mathcal M}_A  \leftrightarrows   {_R\mathcal M}_R : \kappa_{r,*})\] form a pair of adjoint functors  that satisfy the assumptions of  Proposition  \ref{prop:pd}. 
\end{proposition}

\begin{proof}
It is well-known that 	$F_R^l : {_A\mathcal M}  \leftrightarrows   {_R\mathcal M}_R : \kappa_{l,*}$ form a pair of adjoint functors, with the unit and counit of the adjunction being given by the following morphisms, for $V$ a left $A$-module and $M$ an $R$-bimodule
\begin{align*}
	\eta_V : V &\longmapsto \kappa_{l,*}(V\odot R) \quad \quad \varepsilon_M : \kappa_{l,*}(M)\odot R \to M \\
	v &\longmapsto v\otimes 1 \quad \quad \quad \quad  \quad \quad\quad \quad m\otimes x \longmapsto m\cdot x 
\end{align*}
It is clear that $F_R^l$ and $\kappa_{l,*}$ are exact. We have, by \cite[Lemma 2.1]{yu} and its proof, a left $A$-module isomorphism
\begin{align*}
{_AA \otimes R}	&\longrightarrow\kappa_{l,*}(_RR\otimes R_R) \\
a \otimes x &\longmapsto a^{[1] }\otimes xa^{[2]}
\end{align*}
If $V$ is a vector space, the above isomorphism is easily adapted to an $A$-module isomorphism $\kappa_{l,*}(_RR\otimes V \otimes R_R) \simeq {_AA \otimes (V \otimes R)}$, so that $\kappa_{l, *}$ transforms free $R$-bimodules in free $A$-modules, and hence projective  $R$-bimodules into projective $A$-modules: we have $\kappa_{l,*}({\rm Proj}({_R\mathcal M}_R ))\subset {\rm Proj}({_A\mathcal M})$. 
A free $A$-module $_AA\otimes V$ is direct summand of the   free $A$-module $_AA \otimes (V \otimes R)\simeq \kappa_{l,*}(_RR\otimes V \otimes R_R)$, hence belongs to $\mathrm{add}(\kappa_{l,*}(\mathrm{Proj}({_R\mathcal M}_R)))$. Projective modules are themselves direct summands of free $A$-modules, hence we have $\mathrm{Proj}(_A\mathcal M) \subset \mathrm{add}(\kappa_{l,*}(\mathrm{Proj}({_R\mathcal M}_R)))$, so we conclude that $\mathrm{Proj}(_A\mathcal M) = \mathrm{add}(\kappa_{l,*}(\mathrm{Proj}({_R\mathcal M}_R)))$, as required.

The proof in the case of a right Galois object works, similarly, with the right $A$-module isomorphism 
\begin{align*}
	{R \otimes A_A}	&\longrightarrow\kappa_{r,*}(_RR\otimes R_R) \\
	x \otimes a &\longmapsto a^{\langle 1\rangle }x \otimes a^{\langle 2\rangle }
\end{align*}
whose is inverse is conveniently obtained using a cogroupoid $\mathcal{C}$ with two objects $X$, $Y$ such that $R =\Cc(X,Y)$ and $A=\Cc(Y,Y)$ (see \cite{bic14}, the cogroupoid can be obtained by considering fibre functors on left comodules as in \cite{bic14}, or by more direct manipulations as in \cite{gru}). In this setting, the above morphism  is 
\begin{align*}
	\Cc(X,Y) \otimes \Cc(Y,Y)	&\longrightarrow \Cc(X,Y)\otimes \Cc(X,Y)\\
	a^{XY} \otimes b^{YY} &\longmapsto S_{YX}(b^{YX}_{(1)}) a^{XY} \otimes b^{XY}_{(2)}
\end{align*}
and its inverse is 
\begin{align*}
	\Cc(X,Y) \otimes \Cc(X,Y)	&\longrightarrow \Cc(X,Y)\otimes \Cc(Y,Y)\\
	a^{XY} \otimes b^{XY} &\longmapsto S_{YX}S_{XY}(b^{XY}_{(1)}) a^{XY} \otimes b^{YY}_{(2)}
\end{align*}
With this construction, we see that  $\kappa_{r, *}$ transforms free $R$-bimodules in free right $A$-modules, and the rest of the arguments are similar to the left case.
\end{proof}

We now are ready for the comparison of the cohomological dimensions of a Hopf algebra with its Galois objects, with the well-known subtlety that for Galois objects one has to switch from the global dimension to the Hochschild cohomological dimension. Recall that for an algebra $R$, the Hochschild cohomological dimension of $R$, that we denote $\cd(R)$, is the projective dimension of $R$ in the category  ${_R\mathcal M}_R$. It is well known (see the appendix of \cite{wyz} for example) that for a Hopf algebra we have $\gldim(A) = \cd(A)= \pd_A(_\varepsilon k)= \pd_A(k_\varepsilon)$, where  $_\varepsilon k$ and $k_\varepsilon$ are the respective trivial left and right $A$-modules.

\begin{theorem}\label{thm:cdGalois}
	Let $A$ be a Hopf algebra and let $R$ be a left or right $A$-Galois object. If $\cd(A)$ is finite, we have $\cd(A)=\cd(R)$.
\end{theorem}

\begin{proof}
If $R$ is a left (resp. right) $A$-comodule algebra, the canonical functor 	
\[	F_R^l : {_A \mathcal M} \longrightarrow {_R\mathcal M}_R  \quad \quad \text{(resp.} \	F_R^r :  \mathcal M_ A\longrightarrow {_R\mathcal M}_R) \]
sends the trivial left $A$-module $_\varepsilon k$ (resp the trivial right $A$-module $k_\varepsilon)$ to the $R$-bimodule $R$. Hence if  $R$ is left or right $A$-Galois, and $\gldim(A)=\pd_A({_\varepsilon k})=\pd_A(k_\varepsilon)$ is finite,  then by Proposition \ref{prop:adjcan} and Proposition \ref{prop:pd} we have $\gldim(A)= \pd_A({_\varepsilon k})=\pd_A(k_\varepsilon) = \pd_{_R\mathcal M_R}(R)=\cd(R)$.
\end{proof}

\begin{remark}
	In  \cite[Theorem 1.1]{bi24} the conclusion of Theorem \ref{thm:cdGalois} was reached under the additional assumption that $A$ has bijective antipode  and $R$ has a unital  twisted trace with respect to a semi-colinear automorphism. This is of course surpassed by Theorem \ref{thm:cdGalois} above, but we think that it still would be interesting to know whether a Galois object has a unital  twisted trace with respect to a semi-colinear automorphism.
\end{remark}

\begin{remark}
	When $A$ has bijective antipode, the case of right Galois objects in the above theorem can be deduced from the left case,  by noticing that in that case, if $R$ is right $A$-Galois, then $R^{\rm op}$ is left $A$-Galois. The small effort we have made in not deducing the right case from the left case (or vice versa) is what allows us to remove the assumption of the bijectivity of the antipodes in the forthcoming theorem.
\end{remark}


\begin{theorem}
	Let $A$, $B$ be Hopf algebras  such that  $\Mm^A \simeq^\otimes\Mm^B$. If $\gldim(A)$ and $\gldim(B)$ are finite, then $\gldim(A)=\gldim(B)$.
	\end{theorem}
	
	\begin{proof}
		By \cite{sc1} the existence of a monoidal equivalence $\Mm^A \simeq^\otimes\Mm^B$ implies the existence of an $A$-$B$-bi-Galois object $R$, hence by Theorem \ref{thm:cdGalois} we have  $\gldim(A)=\cd(R)=\gldim(B)$ if  $\gldim(A)$ and $\gldim(B)$ are finite.
	\end{proof}

\begin{remark}
Of course Hopf algebras that do not have a bijective antipode are rare and seem to be pathological, so  we do not claim removing	the bijectivity assumption for the antipodes in Theorem \ref{thm:cd0} is a very  important improvement. Nevertheless, Hopf algebras having non bijective antipode exist, and there are some interesting ones, such as the free Hopf algebras generated by matrix coalgebras from \cite{tak71}, for which some non-trivial monoidal equivalences of comodule categories exist, see \cite[Section 5]{bic03}.
\end{remark}

\end{document}